\documentclass[11pt,letterpaper]{article}
\usepackage[utf8]{inputenc} 

\usepackage{amsmath,amssymb,amsfonts} 
\usepackage{epsfig} \usepackage{latexsym,nicefrac,bbm}
\usepackage{xspace}
\usepackage{color,fancybox,graphicx,subfigure,fullpage}
\usepackage{tabularx} \usepackage{hyperref} 
\usepackage{pdfsync}
\usepackage[boxruled, linesnumbered]{algorithm2e}
\usepackage{multicol}
\usepackage{enumitem}
\usepackage{multirow}
 \usepackage{url}

\usepackage{float}

\usepackage{microtype}
\usepackage{graphicx}
\usepackage{subfigure}
\usepackage{booktabs} %
\PassOptionsToPackage{hyphens}{url}
 
\usepackage{hyperref}

\usepackage{amsmath,amssymb,amsfonts} 
\usepackage{epsfig} \usepackage{latexsym,nicefrac,bbm}
\usepackage{xspace}
\usepackage{color,fancybox,graphicx,subfigure}
\usepackage{tabularx} \usepackage{hyperref} 
\usepackage{multicol}
\usepackage{enumitem}
\usepackage{multirow}
 \usepackage{url}

\usepackage{mathtools}
\usepackage{amsthm}
\usepackage{bbm}

\usepackage[capitalize,noabbrev]{cleveref}

\usepackage{url}            
\usepackage{booktabs}       
\usepackage{amsfonts}       
\usepackage{nicefrac}       
\usepackage{microtype}      
\usepackage{xcolor}         

\usepackage[framemethod=TikZ]{mdframed}

\usepackage{booktabs}
\usepackage{microtype}
\usepackage{multicol}
\usepackage{multirow}
\usepackage{tcolorbox}
\usepackage{savesym}

\savesymbol{Bbbk}
\usepackage{amsthm,amsmath,amssymb,amsfonts}
\usepackage{url}
\usepackage{mathrsfs}
\usepackage{nicefrac,bbm}
\usepackage{hyperref, cleveref}
\usepackage{mathtools,pifont,enumitem}
\usepackage[framemethod=TikZ]{mdframed} 
\usepackage{color,fancybox,graphicx,subfigure}

\usepackage{thm-restate}

\savesymbol{st}
\usepackage[normalem]{ulem}
\usepackage{xspace}

\usepackage{xurl}
\usepackage{scalerel}
\usetikzlibrary{math}
\usepackage{bbding}
\usepackage{array}
\usepackage{dsfont}

\newtheorem{theorem}{Theorem}[section]

\newtheorem{lemma}[theorem]{Lemma}
\newtheorem{remark}[theorem]{Remark}

\newtheorem{corollary}[theorem]{Corollary}

\newtheorem{assumption}[theorem]{Assumption}

\renewcommand{\epsilon}{\varepsilon}

\title{Singular Subspace Perturbation Bounds via Rectangular Random Matrix Diffusions}

 \author{Peiyao Lai \\ Worcester Polytechnic Institute \and Oren Mangoubi \\ Worcester Polytechnic Institute}

\begin{document}

\date{}
\maketitle

\bigskip

\begin{abstract}
  Given a matrix $A \in \mathbb{R}^{m\times d}$ with singular values $\sigma_1\geq \cdots \geq \sigma_d$, and a random matrix $G \in \mathbb{R}^{m\times d}$ with iid $N(0,T)$ entries for some $T>0$, we derive  new bounds on the Frobenius distance between subspaces spanned by the top-$k$ (right) singular vectors of $A$ and $A+G$.
This problem arises in numerous applications in statistics where a data matrix may be corrupted by Gaussian noise, and in the analysis of the Gaussian mechanism in differential privacy, where Gaussian noise is added to data to preserve private information.
We show that, for matrices $A$ where the gaps in the top-$k$ singular values are roughly $\Omega(\sigma_k-\sigma_{k+1})$  the expected Frobenius distance between the subspaces is $\tilde{O}(\frac{\sqrt{d}}{\sigma_k-\sigma_{k+1}} \times \sqrt{T})$, improving on previous bounds by a factor of $\frac{\sqrt{m}}{\sqrt{d}} \sqrt{k}$.
To obtain our bounds we view the perturbation to the singular vectors as a diffusion process-- the Dyson-Bessel process-- and use tools from stochastic calculus to track the evolution of the subspace spanned by the top-$k$ singular vectors.
\end{abstract}

\bigskip
\bigskip

\section{Introduction}

Given a matrix $A \in \mathbb{R}^{m\times d}$ with $d \leq m$ and singular values $\sigma_1\geq \cdots \geq \sigma_d$, one oftentimes wishes to approximate the right-singular vectors of $A$ by a lower rank matrix of some rank $k<d$ \cite{shikhaliev2019low, hubert2004robust, james2013introduction, kishore2017literature, liberty2007randomized}.
For instance, one may wish to learn the subspace spanned by the top-$k$ right-singular vectors of $A$, in which case one may seek a projection matrix which minimizes the distance to the projection matrix onto the subspace spanned by the top-$k$ right-singular vectors of $A$. %
One can also consider the related problem of recovering a matrix $\hat{M}_k$ which minimizes the Frobenius distance $\|\hat{M}_k - A^\top A\|_F$ to the covariance matrix $A^\top A$ of the data $A$.
Roughly speaking, both of these problems are instances of the following general problem: given a set of numbers $\gamma_1 \geq \cdots \geq \gamma_d$ and denoting by $\Gamma :=\mathrm{diag}(\gamma_1, \cdots, \gamma_d)$ and by $A = U \Sigma V^\top$ a singular value decomposition of $A$, find a matrix $M \in \mathcal{O}_{\Gamma^2}$ which minimizes the Frobenius distance $\|M - V^\top \Gamma^2 V \|_F$, where $\mathcal{O}_{\Gamma^2} :=\{U\Gamma^2 U^\top: U \in O(d)\}$ denotes the orbit of $\Gamma^2$ under the orthogonal group.
Plugging in $\gamma_i = 1$ for $i \leq k$ and $\gamma_i = 0$ for $i>k$, we recover the problem of finding a projection matrix which minimizes the Frobenius distance to the projection matrix onto the subspace spanned by to top-$k$ right-singular vectors of $A$.
And, roughly speaking, when we set $\gamma_i \approx \sigma_i$ for $i \leq k$ and $\gamma_i = 0$ for $i>k$, we recover the problem of finding a rank-$k$ covariance matrix which minimizes the Frobenius distance to the covariance matrix of $A$.

In many applications, the matrix $A$ is perturbed by a ``noise'' matrix $E \in \mathbb{R}^{m\times d}$ and one only has access to a perturbed matrix $A+E$.
 Oftentimes, the noise matrix consist of iid Gaussian entries.
For instance, in statistics applications, and signal and image processing  applications, this noise may arise as natural background  Gaussian noise obscuring a ``signal'' matrix $A$ \cite{wu1997adaptive, helstrom1955resolution, liu2012additive, djuric1996model, bergmans1974simple}.
 In differential privacy applications, Gaussian noise may be artificially added to the data matrix $A$, or to a machine learning algorithm trained on the data $A$, to hide sensitive information about individuals in the dataset \cite{dwork2006differential, dwork2006our}; see e.g.  \cite{dwork2014analyze, mangoubi2022re, mangoubi2023private}  where {\em symmetric}-matrix Gaussian noise is added to covariance matrices to guarantee privacy.
 The addition of Gaussian noise to ensure privacy is referred to as the Gaussian mechanism, and is known to satisfy $(\epsilon,\delta)$-differential privacy guarantees.

\subsection{Related work}\label{sec_related_work}

 Multiple prior works have shown singular subspace perturbation bounds when $E \in \mathbb{R}^{m\times d}$ may be any (deterministic) matrix. 
For instance, the Davis-Kahan-Wedin sine-Theta theorem \cite{davis1970rotation, wedin1972perturbation} implies a bound of roughly
\begin{equation}\label{eq_DK}
    ||| V_k V_k^\top - \hat{V}_k \hat{V}_k^\top ||| \leq \dfrac{\sqrt{k} ||| E |||}{\sigma_k-\sigma_{k+1}},
\end{equation}
where $V_k, \hat{V}_k$ are, respectively, the matrices whose columns are the top-$k$ right-singular vectors of $A$ and $\hat{A}:= A + E$, and $|||\cdot |||$ is e.g. the Frobenius norm $\| \cdot\|_F$ or the spectral norm $\| \cdot \|_2$.
These bounds are tight (for sufficiently small $|||E |||$) in the general setting where $E \in \mathbb{R}^{m\times d}$ may be any (deterministic) matrix.

 When the perturbation $E$ is, e.g., a Gaussian random matrix with iid $N(0,T)$ entries for some $T>0$, one can plug in high-probability concentration bounds, which imply that $\| E \|_2 \leq O(\sqrt{m}))$ w.h.p., to the deterministic bounds in \eqref{eq_DK}  to obtain a bound of  $$  \| V_k V_k^\top - \hat{V}_k \hat{V}_k^\top \|_F \leq  \dfrac{\sqrt{k} \sqrt{m}}{\sigma_k-\sigma_{k+1}} \times \sqrt{T}$$ w.h.p. 
 However, the resulting bounds may not be tight. 

 Multiple works have obtained tighter bounds than those implied by the deterministic bounds in \eqref{eq_DK}, in different settings when $E$ is a random matrix from some known distribution or class of distributions (see e.g. \cite{o2018random, fan2018ell, abbe2022, cai2021subspace}).
In particular, \cite{o2018random}, show that if  the entries of $E$ satisfy concentration properties which generalize those of Gaussian distributions, and $A$ has rank $r\leq d$, then 
\begin{equation}\label{Vanvu_result}
\| \hat{V}_k \hat{V}_k^\top - V_k V_k^\top  \|_F \le  O\left(k\left(\dfrac{\sqrt{r}}{\sigma_k - \sigma_{k+1}} + \dfrac{m}{\sigma_k(\sigma_k - \sigma_{k+1})} + \dfrac{\sqrt{m}}{\sigma_k}\right)\right)
\end{equation}
w.h.p.

 Subspace perturbation bounds have also been obtained in different settings where the input matrix, and random matrix perturbation, is a {\em symmetric} matrix (see e.g. \cite{dwork2014analyze, eldridge2018unperturbed, fan2018ell}).
For instance, \cite{dwork2014analyze} obtain perturbation bounds for covariance matrices perturbed by symmetric Gaussian noise, and apply these perturbation bounds to a version of the Gaussian mechanism to obtain tighter utility bounds for covariance matrix approximation problems under $(\varepsilon,\delta)$-differential privacy. \cite{mangoubi2022re, mangoubi2023private} improve on some of their utility bounds by viewing the addition of the symmetric Gaussian noise as a symmetric-matrix valued stochastic process, and use tools from and stochastic calculus and random matrix theory to bound the perturbation to the symmetric matrix eigenvectors.

\subsection{Our Contributions}

Given any matrix $A \in \mathbb{R}^{m \times d}$, and a set of numbers $\gamma_1 \geq \cdots \geq \gamma_d$, our main result (Theorem \ref{ThCovariance}) is a bound on the perturbation to the matrix $V^\top \Gamma^2 V \in \mathcal{O}_{\Gamma^2}$ where $A = U \Sigma V^\top$ is a singular value decomposition of $A$. %
 We show that, if the matrix $A$ is perturbed by a matrix $E = \sqrt{T} G$, where $T>0$ and $G$ is a Gaussian random matrix with iid $N(0,1)$ entries, the right-singular vectors $\hat{V} = (\hat{v}_1,\cdots, \hat{v}_d)$ of the perturbed matrix $A+\sqrt{T} G$ satisfy the bound  $$\mathbb{E} \left[\| \hat{V}\Gamma^\top \Gamma\hat{V}^\top - V\Gamma^\top\Gamma V^\top \|_F\right] \le O\left(\sqrt{\sum_{i=1}^k \sum_{j=i+1}^d \dfrac{(\gamma_i^2 - \gamma_j^2)^2}{(\sigma_i - \sigma_j)^2}}  \sqrt{T}\right) ,$$
where the right-hand-side is a sum-of-squares of the ratios of the eigenvalue gaps of $\Gamma$ and $\Sigma$.

Plugging in different values of $\gamma$, we obtain as corollaries bounds for the subspace recovery and low-rank covariance matrix approximation problems.
In particular, show that  $\| V_k V_k^\top - \hat{V}_k \hat{V}_k^\top \|_F \leq O\left(\dfrac{\sqrt{d}}{\sigma_k - \sigma_{k+1}} \sqrt{T}\right)$ whenever the top-$k$ singular value gaps of $A$ are roughly $\Omega(\max(\sigma_k-\sigma_{k+1}, \sqrt{m} \sqrt{T}))$ (Corollary \ref{cor_subspace}).
This improves (in expectation) on the bounds implied by both \cite{davis1970rotation, wedin1972perturbation} and \cite{o2018random} by a factor of roughly $\frac{\sqrt{m}}{\sqrt{d}} \sqrt{k}$ in the above setting where $E$ is a Gaussian random matrix.   
In particular, our bound replaces those bounds' dependence on the number of {\em rows} $m$ with the number of {\em columns} $d$.
This can lead to a large improvement in many applications, as one oftentimes has that the number $m$  of rows in the data matrix (corresponding to the number of datapoints) is much larger than the number of columns  $d$  (which oftentimes correspond to different features in the data).
Our results also imply similar improvements for the low-rank covariance matrix approximation problem (Corrollary \ref{rankkapprox}).

To obtain our bounds, building on several previous works, including \cite{dyson1962brownian, norris1986brownian, bru1989diffusions, mangoubi2022re, mangoubi2023private},
 we view the perturbation of a matrix $A \in \mathbb{R}^{m \times d}$ by Gaussian noise as a Brownian motion on the entries of an $\mathbb{R}^{m \times d}$ matrix, $\Phi(t) := A + B(t)$ where $B(t)$ is a $m \times d$ matrix whose entries undergo iid standard Brownian motions.
This Brownian motion induces a stochastic diffusion process on the singular values and singular vectors of $\Phi(t)$, referred to as the Dyson-Bessel process.
  The evolution of these eigenvalues and eigenvectors is determined by a system of stochastic differential equations  (see e.g. \cite{dyson1962brownian, norris1986brownian, guionnet2021large}).
 This allows us to use Ito's lemma from stochastic calculus to track the evolution of the Frobenius distance as a stochastic integral of a sum-of-squares of perturbations to the (right)-singular vectors of  $\Phi(t)$.
 In particular, the stochastic evolution of the eigenvectors allows us to bypass higher-order matrix derivative terms that arise in Taylor expansions of deterministic perturbations, as these terms vanish in the stochastic derivative when the perturbation is a Brownian motion, due to the independence of random noise additions at each infinitesimal time-step of the Brownian motion.
 This in turn allows us to obtain stronger bounds than would be possible in the deterministic setting.

\section{Main results}\label{main-result}

For any $d>0$, denote by $\mathrm{O}(d)$ the group of orthogonal of $d \times d$ matrices.
For any diagonal matrix $\Lambda \in \mathbb{R}^{d \times d}$, denote by $\mathcal{O}_{\Lambda} :=\{U\Lambda U^\top: U \in O(d)\}$  the orbit of $\Lambda$ under the orthogonal group.

Given any matrix $A \in \mathbb{R}^{m \times d}$, where $d \leq m$, with singular values $\sigma_1\ge ...\ge\sigma_d\ge 0$ and corresponding orthonormal right-singular vectors $v_1,...v_d$, and given any numbers $\gamma_1 \geq \cdots \geq \gamma_d$, our main result (Theorem \ref{ThCovariance}) is a bound on the perturbation to the matrix $V^\top \Gamma^2 V \in \mathcal{O}_{\Gamma^2}$, where $V :=[v_1,...v_d]\in\mathbb{R}^{d\times d}$ and
 $\Gamma := \mathrm{diag}(\gamma_1,\ldots, \gamma_d)$.

Our main result holds under the following assumption on the gaps in the top $k+1$ singular values  $\sigma_1 \ge... \ge \sigma_{k+1}$  of the matrix $A$.
 We note that this assumption is satisfied on many real-world datasets whose singular values exhibit exponential decay (see e.g. Appendix J of \cite{mangoubi2022re} for examples of datasets with exponentially-decaying singular values).

\begin{assumption}[\bf $A, k, T, \sigma, \gamma)$ (Singular value gaps]\label{assumption}
    The gaps in the top $k+1$ singular values $\sigma_1 \ge... \ge \sigma_{k+1}$ of the matrix $A\in \mathbb{R}^{m \times d}$ satisfy $\sigma_i - \sigma_{i+1} \ge 8\sqrt{T}\sqrt{m}
    \log(\frac{1}{\delta})$ for every $i\in[k]$, where $\delta := \frac{1}{8d \gamma_1^2} \times  \frac{\gamma_1^2-\gamma_d^2 }{(\sigma_1-\sigma_d)^2}$. 
\end{assumption}

\noindent We now state our main result. 
\begin{theorem}[\bf Main result]\label{ThCovariance}
Let $T > 0$.
Given a rectangular matrix $A\in\mathbb{R}^{m\times d}$ with singular values $\sigma_1\ge ...\ge\sigma_d\ge 0$ and corresponding orthonormal right-singular vectors $v_1,...v_d$ (and denote $V :=[v_1,...v_d]\in\mathbb{R}^{d\times d}$). 
Let $G$ be a matrix with i.i.d. $N(0,1)$ entries, and consider the perturbed matrix $\hat A := A + \sqrt{T} G \in\mathbb{R}^{m\times d}$. 

Define $\hat\sigma_1 \ge ...\ge \hat\sigma_d \ge 0$ to be the singular values of $\hat A$ with corresponding orthonormal right-singular vectors $\hat v_1,...\hat v_d$ (and denote $\hat V := [\hat v_1,...\hat v_d]$). 

Let $\gamma_1\ge...\ge\gamma_d\ge 0$ and $k\in[d]$ be any numbers such that $\gamma_i=0$ for $i > k$, and define $\Gamma \coloneqq$ diag $(\gamma_1,...,\gamma_d)$.
Then if $A$ satisfies Assumption \ref{assumption} for $(A, k, T, \sigma, \gamma)$, we have
\begin{equation} 
   \mathbb{E} \left[\| \hat{V}\Gamma^\top \Gamma\hat{V}^\top - V\Gamma^\top\Gamma V^\top \|_F^2\right] \le O\left(\sum_{i=1}^k \sum_{j=i+1}^d \dfrac{(\gamma_i^2 - \gamma_j^2)^2}{(\sigma_i - \sigma_j)^2}\right)T.
\end{equation}
\end{theorem}
\noindent We give an overview of the proof of Theorem \ref{ThCovariance} in Section \ref{4section:4}.  The full proof is given in Appendix \ref{sec_full_proofs}.

\subsection{Application to singular subspace recovery.}\label{sec_subspace_recovery} 

To obtain a pertubration bound for the subspace recovery problem, we plug in $\gamma_i=1$ for all $i\le k$, and $\gamma_{i}=0$ for all $i>k$, into Theorem \ref{ThCovariance}.

\begin{corollary}[\bf Subspace recovery] \label{cor_subspace}
\label{Subspace}
Let $T > 0$.
Given a rectangular matrix $A\in\mathbb{R}^{m\times d}$ with singular values $\sigma_1\ge ...\ge\sigma_d\ge 0$ and corresponding right-singular vectors $v_1,...v_d$. 
Let $G$ be a matrix with i.i.d. $N(0,1)$ entries, and consider the perturbed matrix $\hat A = A + \sqrt{T} G$. 

For any $k\in[d]$, define the $d\times k$ matrices $V_k=[v_1,...v_k]$ and $\hat V_k = [\hat v_1,...\hat v_k]$ where $\hat v_1,\cdots , \hat v_k$ denote the right-singular vectors  of $\hat A$ corresponding to its top-$k$ singular values.
Then if $A$ satisfies Assumption \ref{assumption}($A, k, T, \sigma, \gamma$) where $\gamma = (1, \cdots,1,0, \cdots,0)$ is the vector with the first $k$ entries equal to $1$, we have
\begin{equation} \label{subspace-k}
   \mathbb{E}\left[\| \hat{V}_k\hat{V}_k^\top - V_k V_k^\top \|_F\right]\leq O\left(\dfrac{\sqrt{kd}}{\sigma_k - \sigma_{k+1}} \sqrt{T}\right).
\end{equation}
Moreover, if we further have that $\sigma_i - \sigma_{i+1} \ge \Omega(\sigma_k - \sigma_{k+1})$ for all $i \le k$, then
\begin{equation} \label{subspace-d}
   \mathbb{E}\left[\| \hat{V}_k\hat{V}_k^\top - V_k V_k^\top \|_F\right]\leq O\left(\dfrac{\sqrt{d}}{\sigma_k - \sigma_{k+1}} \sqrt{T}\right).
\end{equation}
\end{corollary}
\noindent  The proof of Corollary \ref{Subspace} is given in Appendix \ref{sec_proof_subspace}.

Corollary \ref{Subspace} improves, in the setting where the perturbation $G$ is a Gaussian random matrix, by a factor of $\frac{\sqrt{m}}{\sqrt{d}}$ (in expectation) on the bound $\| \hat{V}_k\hat{V}_k^\top - V_k V_k^\top \|_F \leq O(\dfrac{\sqrt{km}}{\sigma_k - \sigma_{k+1}} \sqrt{T})$ 
w.h.p. implied by the  Davis-Kahan-Wedin sine-Theta theorem \cite{davis1970rotation, wedin1972perturbation}, whenever Assumption \ref{assumption} is satisfied.
If we also have that  $\sigma_i - \sigma_{i+1} \ge \Omega(\sigma_k - \sigma_{k+1})$ for all $i \leq k$ (as is the case for many real-world datasets which may exhibit exponential decay in their singular values), the improvement is $\sqrt{k} \frac{\sqrt{m}}{\sqrt{d}}$

Moreover, 
 Corollary \ref{Subspace} also improves, in the setting where the perturbation $G$ is a Gaussian random matrix, by a factor of $\sqrt{k}\frac{\sqrt{m}}{\sqrt{d}}$, on the bound of $\| \hat{V}_k\hat{V}_k^\top - V_k V_k^\top \|_F \leq O(k \frac{\sqrt{m}}{\sigma_k} \sqrt{T})$  w.h.p. implied by Theorem 18 of \cite{o2018random}, when Assumption \ref{assumption} is satisfied and  e.g. $\sigma_{k}-\sigma_{k+1} = \Omega(\sigma_k)$ (as is also the case for many real-world datasets).
If we also have that  $\sigma_i - \sigma_{i+1} \ge \Omega(\sigma_k - \sigma_{k+1})$ for all $i \leq k$, the improvement is $k\frac{\sqrt{m}}{\sqrt{d}}$.

\subsection{Application to rank-\texorpdfstring{$k$}\  \ covariance matrix approximation.} 

To obtain a perturbation bound for the rank-$k$ covariance matrix approximation problem, we plug in $\gamma_i=\sigma_i$ for all $i\le k$, and $\gamma_{i}=\sigma_{i}$ for all $i>k$, into Theorem \ref{ThCovariance}.

\begin{corollary}[\bf Rank-$k$ covariance matrix approximation]\label{rankkapprox} 
Let $T > 0$.
Given a rectangular matrix $A\in\mathbb{R}^{m\times d}$ with singular values $\sigma_1\ge ...\ge\sigma_d\ge 0$ and with  right-singular vectors $v_1,...v_d$, where we define $V:=[v_1,...v_d]\in\mathbb{R}^{d\times d}$. 
Let $G$ be a matrix with i.i.d. $N(0,1)$ entries, and consider the perturbed matrix that outputs $\hat A = A + \sqrt{T} G$. 

For any $k\in[d]$, define $\Sigma_k \coloneqq$ diag $(\sigma_1,...,\sigma_k,0,...0)$. Define $\hat{\sigma}_1\ge ...\ge\hat{\sigma}_d\ge 0$ to be the singular values of $\hat A$ with corresponding orthonormal right-singular vectors $\hat v_1,...\hat v_d$, where we define $\hat V := [\hat v_1,...\hat v_d]$, and define $\hat{\Sigma}_k \coloneqq$ diag $(\hat{\sigma}_1,...,\hat{\sigma}_k,0,...0)$. 
Then if $A$ satisfies Assumption \ref{assumption} for $(A, k, T, \sigma, \gamma)$ for $\gamma= (\sigma_1,\cdots, \sigma_k, 0 ,\cdots, 0)$,  we have        
\begin{equation} 
       \mathbb{E}\left[\| \hat{V}\hat{\Sigma}_k^\top\hat{\Sigma}_k\hat{V}^\top - V\Sigma_k^\top\Sigma_k V^\top \|_F^2 \right]  \leq O\left( d\|\Sigma_k\|_F^2 + k\sum_{j=k+1}^d \left(\sigma_k \dfrac{\sigma_k}{\sigma_k - \sigma_{j}}\right)^2\right)T.
       \label{rankkbound}
\end{equation}
\end{corollary}
\noindent The proof of Corollary \ref{rankkapprox} is given in Appendix \ref{sec_proof_covariance}.
 In particular, Corollary \ref{rankkapprox} implies that
\begin{align*}
\sqrt{\mathbb{E}\left[\| \hat{V}\hat{\Sigma}_k^\top \hat{\Sigma}_k\hat{V}^\top - V\Sigma_k^\top\Sigma_k V^\top \|_F^2 \right]} 
	& \leq O\left(\sqrt{k}\sqrt{d} \left(\sigma_1 + \sigma_k  \dfrac{\sigma_k}{\sigma_k - \sigma_{k+1}}\right) \right) \sqrt{T}. 
\end{align*}

\noindent Corollary \ref{rankkapprox} improves, in the setting where the perturbation $G$ is a Gaussian random matrix, by a factor of $\frac{\sqrt{m}}{\sqrt{d}}$ (in expectation) on the bound of
$\| \hat{V}\hat{\Sigma}_k^\top \hat{\Sigma}_k\hat{V}^\top - V\Sigma_k^\top\Sigma_k V^\top \|_F \leq O(k^{1.5}\sqrt{m} \sqrt{T}  \sigma_1 + \sigma_k^2 \frac{\sqrt{k} \sqrt{m}}{\sigma_k-\sigma_{k+1}} \sqrt{T})$ w.h.p.  implied by the Davis-Kahan-Wedin sine-Theta theorem \cite{davis1970rotation, wedin1972perturbation} whenever Assumption \ref{assumption} is satisfied (see Appendix \ref{sec_covariance_baselines} for details).
If we also have that $\sigma_{k}-\sigma_{k+1} = \Omega(\sigma_k)$, the improvement is $\frac{\sqrt{m}}{\sqrt{d}}k$.

Moreover, Corollary \ref{rankkapprox} also improves, when the perturbation is Gaussian, by a factor of $\frac{\sqrt{m}}{\sqrt{d}} \sqrt{k}$ (in expectation)  on the bound implied by Theorem 18 of \cite{o2018random} whenever  e.g. $\sigma_{k}-\sigma_{k+1} = \Omega(\sigma_k)$, as in this setting their bound implies $\| \hat{V}\hat{\Sigma}_k^\top \hat{\Sigma}_k\hat{V}^\top - V\Sigma_k^\top\Sigma_k V^\top \|_F  \leq O\left(\sigma_1 k\sqrt{m}\sqrt{T}\right)$ w.h.p. (see Appendix \ref{sec_covariance_baselines} for details).

\begin{remark}[\bf Tightness in full-rank special case]
\label{tightfullrank}
In the special case where $k=d$, we have 
$\|(A+\sqrt{T}G)^\top(A+\sqrt{T}G) - A^\top A\|_F = \| \sqrt{T} A^\top  G + \sqrt{T} G ^\top A + T G^\top G \|_F = \Theta(\| A^\top G\|_F \sqrt{T}) =  \Theta(\|\Sigma_d\|_F \sqrt{d} \sqrt{T})$ 
with high probability.
 Thus, Corollary \ref{rankkapprox} is tight for this special case.

The last equality above holds w.h.p. because $\| A^\top G\|_F^2 = \mathrm{tr}(G^\top A A^\top G) =  \mathrm{tr}(G^\top\Sigma_d \Sigma_d^\top G)
= \mathrm{tr}(\Sigma_d\Sigma_d^\top G G^\top) = \|\Sigma_d\|_F^2 d$ w.h.p.,  where we may assume without loss of generality that $A$ is a diagonal matrix because the distribution of $G$ is invariant w.r.t. multiplication by orthogonal matrices.
\end{remark}

\section{Preliminaries}
\label{4section:3}
In this section, we present preliminary materials used in the proof of our main result.
In particular, we present the aformentioned matrix-valued Brownian motion process $\Phi(t)$ in Section \ref{DysonBesselp}.
Next, we present the stochastic differential equations (SDEs) which govern the evolution of the singular values of right-singular vectors of $\Phi(t)$ in Section \ref{rightsvsde}.

\subsection{Dyson-Bessel process}\label{DysonBesselp}
We consider the matrix-valued stochastic motion process, $\Phi(t)$, where, for all $t \geq 0$, the entries of $\Phi(t)$ evolve as independent standard Brownian motions with  initial condition $\Phi(0) = A$.
In particular, at time $t=T$ we have $\Phi(T) = A + \sqrt{T} G$ where $G$ is an  $m \times d$ Gaussian random matrix with iid $N(0,1)$ entries. 

Recall that $\sigma_1 \geq \ldots \geq \sigma_d$ denote the singular values of $A$. 
At every time $t > 0$, we denote (with slight abuse of notation) the singular values of $\Phi(t)$  by $\sigma_1(t)\geq \sigma_2(t) \geq \dots \geq \sigma_d(t)$.
 In particular $\sigma_i \equiv \sigma_i(0)$ for all $i \in [d]$, and the singular values  $\sigma_1(t), \ldots, \sigma_d(t)$ are distinct at every time $t>0$ with probability 1 (see e.g. \cite{guionnet2021large}).
 The matrix-valued Brownian motion $\Phi(t)$ induces stochastic diffusion processes on the singular values $\sigma_i(t)$ and singular vectors $v_i(t)$, referred to as the Dyson-Bessel process. 
The dynamics of the singular values $\sigma_i(t)$ of the Dyson-Bessel process are given by the following system of stochastic differential equations  (see e.g. \cite{norris1986brownian} or Theorem 1 in \cite{bru1989diffusions}),
\begin{equation}\label{eq:sde_singular_s}
    \mathrm{d} \sigma_i(t) = \mathrm{d}\beta_{ii}(t) + \left(\frac{1}{2\sigma_i(t)}\sum_{\{j\in [d] : j \neq i\}}\frac{(\sigma_i(t))^2 + (\sigma_j(t))^2}{(\sigma_i(t))^2 - (\sigma_j(t))^2}+ \frac{m-1}{2 \sigma_i(t)} \right)\mathrm{d} t, \qquad \forall  1\leq i\leq d,
\end{equation}
where $\beta_{ii}, 1\le i\le d$ is a family of independent one-dimensional Brownian motions.

\subsection{Right singular vector SDE}\label{rightsvsde}
The dynamics of right-singular vectors $v_i(t)$ of the Dyson-Bessel process are governed by the following stochastic differential equations (see e.g. \cite{norris1986brownian} or Theorem 2 in \cite{bru1989diffusions}),
\begin{align}
    \mathrm{d} v_i(t) & = \sum_{\{j\in [d] : j \neq i\}} v_j (t) \sqrt{\dfrac{(\sigma_j(t))^2 + (\sigma_i(t))^2}{((\sigma_j(t))^2 - (\sigma_i(t))^2)^2}} \mathrm{d}\beta_{ji}(t) - \frac{v_i (t)}{2} \sum_{\{j\in [d] : j \neq i\}} \dfrac{(\sigma_j(t))^2 +(\sigma_i(t))^2}{((\sigma_j(t))^2 - (\sigma_i(t))^2)^2} \mathrm{d}t\nonumber\\
    & = \sum_{\{j\in [d] : j \neq i\}} v_j (t) c_{ij} (t) \mathrm{d}\beta_{ji}(t) -  \frac{v_i (t)}{2} \sum_{\{j\in [d] : j \neq i\}} c_{ij}^2 (t) \mathrm{d}t,  \qquad \forall  1\leq i\leq d, \label{eq:sde_singular_v}
\end{align}
where $\beta_{ij}(t), 1 \leq i < j \leq d$, is a family of independent standard one-dimensional Brownian motions, and the $\beta_{ij}(t)$ form a skew-symmetric matrix, i.e. $\beta_{ij}(t) = - \beta_{ji}(t)$ for all $t \geq 0$.
For convenience, in the above equation, we denote $c_{ij} (t) := \sqrt{\frac{(\sigma_j(t))^2 + (\sigma_i(t))^2}{((\sigma_j(t))^2 - (\sigma_i(t))^2)^2}} = c_{ji} (t)$ for all $i, j \in [d]$.

\subsection{Ito's Lemma}

We will also use the following result from stochastic Calculus, Ito's Lemma, which is a generalization of the chain rule in deterministic calculus.
\begin{lemma}[\bf Ito's Lemma \cite{ito1951formula}]\label{Lemma_Ito}
Let $f: \mathbb{R}^d \rightarrow \mathbb{R}$  be a second-order differentiable function, and let $X(t)$ be a diffusion process on $\mathbb{R}^d$.
 Then
 \begin{equation*}
\textstyle \mathrm{d} f(X_t) =  (\nabla f(X_t))^\top \mathrm{d} X_t + \frac{1}{2} (\mathrm{d} X_t)^\top(\nabla^2 f(X_t)) \mathrm{d}X_t  \qquad \forall    t \geq 0.
 \end{equation*}
\end{lemma}
\bigskip

\subsection{Other preliminaries}\label{preliminaries}

We will use the following deterministic eigenvalue perturbation bound
\begin{lemma}[\bf Weyl's Inequality \cite{weyl1912asymptotische}]\label{lemma_weyl} Let  $A, E \in \mathbb{R}^{m \times d}$ is a  matrix.
Denote by $\sigma_1 \geq \ldots \geq \sigma_d$ the singular values of $A$ and by $\hat{\sigma}_1 \geq \ldots \geq \hat{\sigma}_d$ the singular values of $A+ E$.
Then
    \begin{equation*}
  |\sigma_i - \hat{\sigma}_i | \le \| E \|_2 \qquad \forall i \in [d].
    \end{equation*}
\end{lemma}

The following concentration bound, Theorem 4.4.5 of  \cite{vershynin2018high} applied  to Gaussian random matrices, will allow us to bound the spectral norm of the Gaussian perturbation $G$ (which in turn will allow us to apply \eqref{lemma_weyl} to bound the perturbations to eigenvalues).
\begin{lemma}[\bf Spectral-norm concentration bound for Gaussian matrices \cite{vershynin2018high}]\label{lemma_concentration}
If $G \in \mathbb{R}^{m \times d}$ is a Gaussian random matrix with iid $N(0,1)$ entries, then
$$\mathbb{P}(\|G\|_2 > \sqrt{m} + \sqrt{d} + s) < 2e^{-s^2} \qquad \forall s >0.$$
\end{lemma}

\section{Overview of proof of Theorem \ref{ThCovariance}}
\label{4section:4}

We present an overview of the proof of Theorem \ref{ThCovariance} along with the main technical lemmas used in the proof.
In Steps 1 and 2 we express the perturbed matrix, and its quantities of interest derived from its right-singular vectors, as matrix-valued diffusions.
Steps 3, 4, and 5 present the main technical lemmas, and we complete the proof in Step 6. 
The full proof is given in Appendix \ref{main-proof}.

\subsection{Viewing the perturbed matrix as a matrix-valued Brownian motion.} 
To obtain our bounds, we begin by defining  the matrix-valued Brownian motion, $$\Phi(t) := A + B(t) \qquad \qquad \forall t \geq 0,$$ where the entries of $B(t)$ evolve as independent standard Brownian motions initialized at $0$.
In particular, at time $t=0$ we have $\Phi(0) = A$, and at time $t=T$ we have $\Phi(T) = A + \sqrt{T} G$ where $G$ is an  $m \times d$ Gaussian random matrix with iid $N(0,1)$ entries.

\subsection{Projecting the matrix Brownian motion onto the orthogonal orbit $\mathcal{O}_{\Gamma^2}$. }
        Denote by $A = U\Sigma V^\top$ and $\hat A = \hat U\hat \Sigma\hat V^\top$ singular value decompositions of $A$ and $\hat A$, respectively, where $U, \hat{U} \in \mathrm{O}(m)$, $V, \hat{V} \in \mathrm{O}(d)$, and $\Sigma, \hat{\Sigma} \in \mathbb{R}^{m \times d}$ are diagonal.
        
 	Recall that our goal is to bound the quantity $\mathbb{E}[\| \hat{V}\Gamma^\top \Gamma\hat{V}^\top - V\Gamma^\top\Gamma V^\top \|_F]$, where $A^\top A = V \Sigma^\top\Sigma V^\top$ and $\hat{A}^\top \hat{A} = \hat{V} \hat{\Sigma}^\top\hat{\Sigma} \hat{V}^\top$ are eigenvalue decompositions of $A^\top A $ and $\hat{A}^\top \hat{A} $. 
     To obtain a bound on this quantity,
     we first define a stochastic process $\Psi(t)$ for which  $\Psi(0) = V \Gamma^\top \Gamma V^\top$ and $\Psi(T) = \hat{V} \Gamma^\top \Gamma \hat{V}^\top$. 
     We then bound the expected Frobenius distance $$\mathbb{E}[\| \hat{V}\Gamma^\top \Gamma\hat{V}^\top - V\Gamma^\top\Gamma V^\top \|_F] = \mathbb{E}[\| \Psi(T) - \Psi(0) \|_F]$$ by integrating the stochastic derivative of $\Psi(t)$ over the time period $[0, T]$.
	
	Towards this end, at every time $t \geq 0$, define $\Phi(t) := U(t)\Sigma (t) V(t)^\top$ to be a singular value decomposition of the rectangular matrix $\Phi(t)$, where $\Sigma(t) \in \mathbb{R}^{m \times d}$ is a diagonal matrix whose diagonal entries are the singular values $\sigma_1(t) \geq \cdots \ge \sigma_d(t)$ of $\Phi(t)$.
 $V(t) = [v_1(t), \cdots, v_d(t)]$ is a $d\times d$ orthogonal matrix whose columns $v_1(t), \cdots, v_d(t)$ are the corresponding right-singular vectors of $\Phi(t)$. 
 $V(t) \in \mathrm{O}(m)$ is an $m\times m$ orthogonal matrix whose columns are left-singular vectors of $\Phi(t)$.
    
     At every time, denote by $\Psi(t) \in \mathcal{O}_{\Gamma^2}$ to be the symmetric matrix with given eigen values $\Gamma(t)^\top \Gamma(t)$ and eigenvectors given by the columns of $V(t)$: $$\Psi(t) :=V(t)\Gamma(t)^\top \Gamma(t) V(t)^\top, \forall t \in [0,T].$$
       In other words, $\Psi(t) \in \mathcal{O}_{\Gamma^2}$ is the Frobenius-distance minimizing projection of the matrix Brownian motion $\Phi(t)$ onto the orthogonal orbit manifold $\mathcal{O}_{\Gamma^2}$.
	
\subsection{Deriving an expression for the stochastic derivative $\mathrm{d}\Psi(t)$.}
	To bound the expected squared Frobenius distance $\mathbb{E}\left[\|\Psi(T) - \Psi(0)\|_F^2\right]$ we would like to express it as an integral in terms of the stochastic derivative of $\Phi(t)$.
 
Towards this end, we use the stochastic differential equations which govern the evolution of the eigenvectors of the Dyson-Bessel process \eqref{eq:sde_singular_v} to derive an expression for the stochastic derivative $\mathrm{d}\Psi(t)$ of the matrix diffusion $\Psi(t)$ (Lemma \ref{sdeofpsi}),
    \begin{align}\label{eq_SD_1}
    \mathrm{d}\Psi(t) &=\sum_{i=1}^d \gamma_i^2 \mathrm{d}(v_i (t) v_i^\top (t))\nonumber\\
    &\sum_{i=1}^d \sum_{j \neq i} (\gamma_i^2 - \gamma_j^2 ) \left[ \dfrac{c_{ij} (t) }{2}\mathrm{d}\beta_{ji} (t) ( v_i (t)  v_j^\top (t) + v_j (t) v_i^\top (t) )-  c_{ij}^2 (t)\mathrm{d}t    (v_i (t) v_i^\top (t))\right].
    \end{align}
		
\subsection{Using Gaussian concentration and Weyl's inequality to bound the singular value gaps.}
	The above equation \eqref{eq_SD_1} for the stochastic derivative $\mathrm{d}\Psi(t)$ includes terms $c_{ij}(t)$, whose magnitude is proportional to the inverse of the gaps in the squared singular values $\sigma_i^2(t) - \sigma_j^2(t)$ for each $i, j \in[d]$. 
 In order to bound these terms, we use Weyl's inequality \eqref{lemma_weyl} together with standard concentration bounds for the spectral norm of Gaussian random matrices (Lemma \ref{lemma_concentration}), to show that the gaps $\sigma_i(t) - \sigma_{j}(t)$ in the top $k+1$ singular values  satisfy (Lemma \ref{boundofgap}),
	$$\sigma_i (t) - \sigma_j(t) \ge \dfrac{1}{2}(\sigma_i - \sigma_j) \qquad \forall t \in [0,T], \, \, i<j\leq k+1$$
	with high probability at least $1-\delta$, provided that the initial gaps are sufficiently large to satisfy  Assumption \ref{assumption}($A, k, T, \sigma, \gamma$) .

 This implies that, with high probability at least $1-\delta$, the inverse-eigenvalue gap terms in \eqref{eq_SD_1} satisfy (Lemma \ref{boundofc})
\begin{align}\label{eq_boundofc_2}
    c_{ij} (t) &= \sqrt{\frac{(\sigma_j(t))^2 + (\sigma_i(t))^2}{((\sigma_j(t))^2 - (\sigma_i(t))^2)^2}} \nonumber\\
    &\leq \dfrac{4}{\sigma_i - \sigma_j},\qquad \forall i<j,  \, \, t \in [0,T]. 
\end{align}
 
\subsection{Integrating the stochastic derivative of d$\Psi(t)$ over the time interval $[0,T]$.}
	Next we express the expected squared Frobenius distance $\mathbb{E}\left[\|\Psi(T) - \Psi(0)\|_F^2\right]  $ as an integral $\mathbb{E}\left[\|\Psi(T) - \Psi(0)\|_F^2\right]  =\mathbb{E}\left[ \|\int_0^T \mathrm{d}\Psi(t)\|_F^2 \right] $.
 
 Next, we apply Ito's Lemma (Lemma \ref{Lemma_Ito}) to  $f(\Psi(t))$ where $f(X) := \| \cdot \|_F^2$, and plug in our high-probability bound on the inverse eigenvalue gap terms $c_{ij}(t)$ \eqref{eq_boundofc_2}, to derive an upper bound for the integral $\mathbb{E}\left[ \|\int_0^T \mathrm{d}\Psi(t)\|_F^2 \right]$, which gives roughly (Lemma \ref{boundofz})
	 \begin{align}\label{eq:z}
	      &\mathbb{E}\left[\|\Psi(T) - \Psi(0)\|_F^2\right]   \nonumber \\
     &\leq \int_0^T \mathbb{E}\left[ \sum_{i=1}^d \sum_{j \neq i}   (\gamma_i^2 - \gamma_j^2 )^2c_{ij}^2(t) \right] \mathrm{d}t + T \int_0^T \mathbb{E}\left[ \sum_{i=1}^d \left(\sum_{j\neq i} (\gamma_i^2 - \gamma_j^2 )^2c_{ij}^2(t)\right)^2\right]\mathrm{d}t \nonumber\\        
     &\le \int_0^T \mathbb{E}\left[ \sum_{i=1}^d \sum_{j \neq i}   \dfrac{(\gamma_i^2 - \gamma_j^2 )^2}{(\sigma_i - \sigma_j)^2} \right] \mathrm{d}t + T \int_0^T \mathbb{E}\left[ \sum_{i=1}^d \left(\sum_{j\neq i} \dfrac{(\gamma_i^2 - \gamma_j^2 ) }{(\sigma_i - \sigma_j)^2}\right)^2\right]\mathrm{d}t.  
\end{align}

	 Noting that the second term on the right-hand side of (\ref{eq:z})is at least as small as the first term, and applying the Cauchy-Schwarz inequality to the second term, we get that (Theorem \ref{ThCovariance}),
  \begin{align*} 
   \mathbb{E} \left[\| \hat{V}\Gamma^\top \Gamma\hat{V}^\top - V\Gamma^\top\Gamma V^\top \|_F^2\right] &= \mathbb{E}\left[\|\Psi(T) - \Psi(0)\|_F^2\right]\\
   &\le O\left(\sum_{i=1}^k \sum_{j=i+1}^d \dfrac{(\gamma_i^2 - \gamma_j^2)^2}{(\sigma_i - \sigma_j)^2}\right)T.
\end{align*}

\section{Conclusion}

  In this paper, we obtain Frobenius-norm bounds on the perturbation to the singular subspace spanned by the top-$k$ singular vectors of a matrix $A \in \mathbb{R}^{m\times d}$, when $A$ is perturbed by an $m\times d$ Gaussian random matrix. 
  Our bounds improve, in many settings where the perturbation is Gaussian, on bounds implied by previous works, by a factor of roughly $\frac{\sqrt{m}}{\sqrt{d}} \sqrt{k}$.
  This may lead to a large improvement in many applications, as one oftentimes has that the number $m$  of rows in the data matrix (corresponding to the number of datapoints) is much larger than the number of columns  $d$  (which oftentimes correspond to different features in the data).
To obtain our bounds we view use tools from stochastic calculus to track the evolution of the subspace spanned by the top-$k$ singular vectors.

 On the other hand, we note that our bounds assume that the top-$k$ singular value gaps of $A$ are roughly $\Omega(\sqrt{m})$; while this assumption may hold in settings where the data matrix has fast-decaying singular values, it would be interesting to see if it is possible to relax this assumption.
 Moreover, we note that our bounds only apply in the special case when the perturbation $G$ is Gaussian, and it would be interesting to see whether our bounds can be extended to other random matrix distributions.

\newpage

\tableofcontents

\newpage

\appendix

\section{Proof Theorem \ref{ThCovariance}}\label{sec_full_proofs} 
\subsection{Proof of Lemma \ref{sdeofv}}

We first decomposite the matrix $\Psi(t)$ as a sum of its right-singular vectors: $\Psi(t) = \sum_{i=1}^d \gamma_i^2 (v_i (t) v_i^\top (t))$. 
Thus we have 
\begin{equation}
    \mathrm{d}\Psi(t) = \sum_{i=1}^d \gamma_i^2 \mathrm{d}(v_i (t) v_i^\top (t))
\end{equation}
We begin by computing the stochastic derivative $\mathrm{d}v_i (t) v_i^\top (t)$ for each $i\in [d]$, by applying the formula in (\ref{eq:sde_singular_v}), together with Ito's Lemma (Lemma \ref{Lemma_Ito}).

\begin{lemma}[Stochastic derivative of $v_i(t)v_i(t)^\top$]\label{sdeofv}
    For all $t\in[0,T]$, 
    $$ \mathrm{d}\left(v_i (t)v_i^\top (t)\right) = \sum_{j \neq i} v_j (t) c_{ij} (t) \mathrm{d}\beta_{ji}(t) - \frac{1}{2} v_i (t) \sum_{j \neq i} c_{ij}^2 (t) \mathrm{d}t. $$
\end{lemma}
\begin{proof} The dynamic of right-singular vectors \cite{bru1989diffusions} are the following:
\begin{align*}
    \mathrm{d} v_i(t) & = \sum_{j \neq i} v_j (t) \sqrt{\dfrac{\lambda_{j}(t) + \lambda_{i}(t)}{(\lambda_{j}(t) - \lambda_{i}(t))^2}} \mathrm{d}\beta_{ji}(t) - \frac{1}{2} v_i (t) \sum_{j \neq i} \dfrac{\lambda_{j}(t) + \lambda_{i}(t)}{(\lambda_{j}(t) - \lambda_{i}(t))^2} \mathrm{d}t\nonumber\\
    & = \sum_{j \neq i} v_j (t) c_{ij} (t) \mathrm{d}\beta_{ji}(t) - \frac{1}{2} v_i (t) \sum_{j \neq i} c_{ij}^2 (t) \mathrm{d}t. \label{eq:sde_singular_v}
\end{align*}
Thus, we have
\begin{align}
    & \mathrm{d} \left(v_i (t)v_i^\top (t)\right)  = \left(v_i (t) + \mathrm{d}v_i (t)\right)\left(v_i (t) + \mathrm{d}v_i (t)\right)^\top-v_i (t) v_i^\top (t)
    \nonumber\\
    & = \left(v_i(t)  + \sum_{j \neq i} v_j(t) c_{ij}(t) \mathrm{d}\beta_{ji}(t) - \frac{1}{2} v_i(t) \sum_{j \neq i} c_{ij}^2  \mathrm{d}t\right)\nonumber\\
    & \qquad \qquad \times \left(v_i(t)^\top + \sum_{j \neq i} v_j(t)^\top  c_{ij}(t) \mathrm{d}\beta_{ji}(t) - \frac{1}{2} v_i(t)^\top \sum_{j \neq i} c_{ij}^2(t) \mathrm{d}t\right)- v_i(t)  v_i(t)^\top  \nonumber\\
    & = v_i (t) \left(\sum_{j \neq i} v_j^\top (t) c_{ij} (t) \mathrm{d}\beta_{ji}(t)\right) - \frac{1}{2}v_i (t) v_i^\top (t) \sum_{j \neq i} c_{ij}^2 (t) \mathrm{d}t + \left(\sum_{j \neq i} v_j (t) c_{ij} (t) \mathrm{d}\beta_{ji}(t)\right)v_i^\top (t)  \nonumber\\
    & + \left(\sum_{j \neq i} v_j (t) c_{ij} (t) \mathrm{d}\beta_{ji}(t)\right)\left(\sum_{j \neq i} v_j^\top (t) c_{ij} (t) \mathrm{d}\beta_{ji}(t)\right) - \frac{1}{2}v_i (t) v_i^\top (t) \sum_{j \neq i} c_{ij}^2 (t) \mathrm{d}t + o(\mathrm{d}t)  \nonumber\\
    &  = v_i (t) \left(\sum_{j \neq i} v_j^\top (t) c_{ij} (t) \mathrm{d}\beta_{ji}(t)\right) + \left(\sum_{j \neq i} v_j (t) c_{ij} (t) \mathrm{d}\beta_{ji}(t)\right)v_i^\top (t) - v_i (t) v_i^\top (t) \sum_{j \neq i} c_{ij}^2 (t) \mathrm{d}t \nonumber\\
    & + \sum_{k\neq i} \sum_{j\neq i} v_k (t) v_j^\top (t)c_{ik} (t) c_{ij} (t) \mathrm{d}\beta_{ki}(t) \mathrm{d}\beta_{ji}(t)                                  \nonumber\\
    &  = v_i (t) \left(\sum_{j \neq i} v_j^\top (t) c_{ij} (t) \mathrm{d}\beta_{ji}(t)\right) + \left(\sum_{j \neq i} v_j (t) c_{ij} (t) \mathrm{d}\beta_{ji}(t)\right)v_i^\top (t) - v_i (t) v_i^\top (t) \sum_{j \neq i} c_{ij}^2 (t) \mathrm{d}t \nonumber\\
   & + \sum_{k\neq i} \sum_{j\neq i} v_k (t) v_j^\top (t)c_{ik} (t) c_{ij} (t) \mathbbm{1}_{\{(kj)= (ii)\}}\mathrm{d}t                                 \nonumber\\
    & = \sum_{j \neq i} c_{ij} (t)  \mathrm{d}\beta_{ji} (t) ( v_i (t) v_j^\top (t) + v_j (t) v_i^\top (t) ) - \sum_{j\neq i}  c_{ij}^2 (t)\mathrm{d}t ( v_i (t) v_i^\top (t) - v_j (t) v_j^\top (t) ). \nonumber
\end{align}
\end{proof}

\subsection{Proof of Lemma \ref{sdeofpsi}}
Recall that 
\begin{equation*}
    \Psi(t) = \sum_{i=1}^d \gamma_i^2 (v_i (t) v_i^\top (t)).
\end{equation*}

We now apply Lemma \ref{sdeofv} to compute the stochastic derivative of $\Psi(t)$.

\begin{lemma}[Stochastic derivative of $\Psi(t)$]\label{sdeofpsi}
     For all $t\in[0,T]$, we have that 
     $$ \mathrm{d}\Psi(t) = \sum_{i=1}^d \sum_{j \neq i} (\gamma_i^2 - \gamma_j^2 ) \left[ \dfrac{c_{ij} (t) }{2}\mathrm{d}\beta_{ji} (t) ( v_i (t)  v_j^\top (t) + v_j (t) v_i^\top (t) )-  c_{ij}^2 (t)\mathrm{d}t    (v_i (t) v_i^\top (t))\right].  $$
\end{lemma}
\begin{proof}   
\begin{align}
	\mathrm{d}\Psi(t) &= \sum_{i=1}^d \gamma_i^2 \mathrm{d}(v_i (t) v_i^\top (t))\nonumber\\
	=& \sum_{i=1}^d \gamma_i^2 \left( \sum_{j \neq i} c_{ij} (t)  \mathrm{d}\beta_{ji} (t) ( v_i (t) v_j^\top (t) + v_j (t) v_i^\top (t) ) - \sum_{j\neq i}  c_{ij}^2 (t)\mathrm{d}t ( v_i (t) v_i^\top (t) - v_j (t) v_j^\top (t) ) \right)\nonumber\\
    =& \sum_{i=1}^d \sum_{j \neq i} \gamma_i^2  c_{ij} (t) \mathrm{d}\beta_{ji} (t) (  v_j (t) v_i^\top (t) + v_i (t) v_j^\top (t) ) - \sum_{i=1}^d \sum_{j\neq i}  \gamma_i^2 c_{ij}^2 (t)\mathrm{d}t ( v_i (t) v_i^\top (t) - v_j (t) v_j^\top (t) )  \nonumber\\
    =& \dfrac{1}{2}\sum_{i=1}^d \sum_{j \neq i} (\gamma_i^2 - \gamma_j^2 )  c_{ij} (t)  \mathrm{d}\beta_{ji} (t) ( v_i  v_j(t)^\top + v_j(t)  v_i^\top )\nonumber\\
    &\qquad \qquad- \dfrac{1}{2}\sum_{i=1}^d \sum_{j\neq i} (\gamma_i^2 - \gamma_j^2 ) c_{ij}^2 (t)\mathrm{d}t   ( v_i  v_i^\top  - v_j(t)  v_j(t)^\top  ) \nonumber\\
    = & \dfrac{1}{2}\sum_{i=1}^d \sum_{j \neq i} (\gamma_i^2 - \gamma_j^2 )  c_{ij} (t)  \mathrm{d}\beta_{ji} (t) ( v_i  v_j(t)^\top + v_j(t)  v_i(t)^\top )- \sum_{i=1}^d \sum_{j\neq i} (\gamma_i^2 - \gamma_j^2 ) c_{ij}^2 (t)\mathrm{d}t    v_i  v_i^\top ,\nonumber
\end{align}
note that the last two equations are because of these observations:
\begin{align*}
     c_{ij} (t) \mathrm{d}\beta_{ij} (t) (  v_j(t)  v_i(t)^\top + v_i(t)  v_j(t)^\top ) & = - c_{ij} (t) \mathrm{d}\beta_{ji} (t) (  v_j(t) v_i(t)^\top + v_i(t)  v_j(t)^\top )   \\
     c_{ij}^2 (t)\mathrm{d}t   ( v_i(t)  v_i(t)^\top  - v_j(t)  v_j(t)^\top  ) & = - c_{ij}^2 (t)\mathrm{d}t   ( v_j(t)  v_j(t)^\top  - v_i(t)  v_i(t)^\top  )    \\
     (\gamma_i^2 - \gamma_j^2 ) c_{ij}^2 (t)\mathrm{d}t ( v_i(t)  v_i(t)^\top  - v_j(t)  v_j(t)^\top  ) & =(\gamma_j^2 - \gamma_i^2 ) c_{ij}^2 (t)\mathrm{d}t   ( v_j(t)  v_j(t)^\top  - v_i(t)  v_i(t)^\top  ).
\end{align*}
\end{proof}

\subsection{Proofs of Lemmas \ref{boundofgap} and \ref{boundofc}}
Next, we show high-probability bounds on the singular gaps $\sigma_i (t) - \sigma_j (t)$ (Lemma \ref{boundofgap}) and coefficients $c_{ij}(t)$ (Lemma \ref{boundofc}).
\begin{lemma}[Bound on singular gaps:]\label{boundofgap}
Suppose that Assumption \ref{assumption} for $(A, k, T, \sigma, \gamma)$  is satisfied.  Then for all $t\in[0,T]$, with probability  $1-\delta$ where $\delta := \frac{1}{8d \gamma_1^2} \times  \frac{\gamma_1^2-\gamma_d^2 }{(\sigma_1-\sigma_d)^2}$, we have
      $|\sigma_i (t) - \sigma_j (t)| \ge \dfrac{1}{2}(\sigma_i - \sigma_j)$ for any $i<j$.
\end{lemma}
\begin{proof}
With probability  $1-\delta$, by Lemma \ref{lemma_concentration}, we have $\|G\|_2 
\leq  2\sqrt{\max\{m, d\}} \log(\frac{1}{\delta})= 2\sqrt{m}\log(\frac{1}{\delta})$. We know $| \sigma_i(t) - \sigma_i | \le \sigma_i + \|G\|_2 = \sigma_i + 2\sqrt{m}\log(\frac{1}{\delta}) $ for any $i$, 
therefore, we bound $|\sigma_i (t) - \sigma_j (t)| = |\sigma_i(t) - \sigma_j(t)| \ge \sigma_i - \sigma_j - 4\sqrt{m}\log(\frac{1}{\delta}) \ge \dfrac{1}{2}(\sigma_i - \sigma_j)$ for any $i<j$ and any $t\in[0,T]$.
\end{proof}

The following proposition shows that the symmetric coefficients $c_{ij} (t)$ are bounded by the reciprocal of the initial singular value gaps.
\begin{lemma}[Bound of coefficients $c_{ij} (t)$]\label{boundofc} 
Suppose that Assumption \ref{assumption} for $(A, k, T, \sigma, \gamma)$  is satisfied.  Then for all $t\in[0,T]$, with probability  $1-\delta$ where $\delta := \frac{1}{8d \gamma_1^2} \times  \frac{\gamma_1^2-\gamma_d^2 }{(\sigma_1-\sigma_d)^2}$, we have
\begin{equation*}
    c_{ij} (t) \le \dfrac{4}{\sigma_i - \sigma_j},\quad\text{for any $i<j$}. 
\end{equation*}
\end{lemma}
\begin{proof} By Lemma \ref{boundofgap}, we have we have with probability at least $1-\delta$ 
    \begin{align}
    c_{ij} (t) &= \dfrac{\sqrt{\sigma_{j}^2(t) + \sigma_{i}^2(t)}}{|\sigma_{j}^2(t) - \sigma_{i}^2(t)|} \nonumber \\
    & \le 2 \dfrac{\sigma_j(t) + \sigma_i(t)}{|\sigma_{j}(t) - \sigma_{i}(t)|(\sigma_i(t) + \sigma_i(t))} \nonumber \\
    & = \dfrac{2}{|\sigma_{j}(t) - \sigma_{i}(t)|}  = \dfrac{2}{|\sigma_i (t) - \sigma_j (t)|}  \le \dfrac{4}{\sigma_i - \sigma_j},\quad \text{for any $i<j$.}   \nonumber 
\end{align}
\end{proof}

\subsection{Proof of Lemma \ref{boundofz}}

Next, to bound the quantity $\mathbb{E}\left[\Vert \Psi(T) - \Psi(0)\Vert_F^2\right]$, use Lemma \ref{sdeofpsi} together with Ito's Lemma (Lemma \ref{Lemma_Ito}), and then apply Lemma \ref{boundofc} to the resulting expression (Lemma \ref{boundofz}).
\begin{lemma}[Bound the Frobenius error as an integral of $\Psi(t)$]\label{boundofz}
    \begin{align}\label{eq:boundZ}
     &\mathbb{E}\left[\Vert \Psi(T) - \Psi(0)\Vert_F^2\right]   \nonumber \\
     &\le 64\int_0^T \mathbb{E}\left[ \sum_{i=1}^d \sum_{j \neq i}   \dfrac{(\gamma_i^2 - \gamma_j^2 )^2}{(\sigma_i - \sigma_j)^2}  \right]\mathrm{d}t  + 32T \int_0^T \mathbb{E}\left[ \sum_{i=1}^d \left(\sum_{j\neq i} \dfrac{(\gamma_i^2 - \gamma_j^2 ) }{(\sigma_i - \sigma_j)^2}\right)^2 \right]\mathrm{d}t.  
\end{align}
\end{lemma}
\begin{proof}

Let $E$ be the event that $|\sigma_i (t) - \sigma_j (t)| \ge \dfrac{1}{2}(\sigma_i - \sigma_j)$ for any $i<j$ and any $t \in [0,T]$. 
By Lemma \ref{boundofgap}, we have $\mathbb{P}(E) \geq 1- \delta$.

We have 
\begin{align}\label{eq_a4}
    &\| \Psi(T) -\Psi(0)\|_F^2 = \|\int_o^T \mathrm{d} \Psi(t)\|_F^2  \nonumber \\
    & = \|\dfrac{1}{2}\int_0^T\sum_{i=1}^d \sum_{j \neq i} |\gamma_i^2 - \gamma_j^2 | |c_{ij} (t)|  \mathrm{d}\beta_{ji} (t)  ( v_i(t)  v_j(t)^\top + v_j(t)  v_i(t)^\top )\nonumber\\
    &\qquad \qquad- \int_0^T\sum_{i=1}^d \sum_{j\neq i}  (\gamma_i^2 - \gamma_j^2 ) c_{ij}^2 (t)\mathrm{d}t v_i(t)  v_i(t)^\top\|_F^2 \nonumber \\
        & \leq 3\left\|\dfrac{1}{2}\int_0^T\sum_{i=1}^d \sum_{j \neq i} |\gamma_i^2 - \gamma_j^2 | |c_{ij} (t)|  \mathrm{d}\beta_{ji} (t)  ( v_i(t)  v_j(t)^\top + v_j(t)  v_i(t)^\top )\right\|_F^2 \nonumber\\
        &\qquad \qquad \qquad+ 3\left \|\int_0^T\sum_{i=1}^d \sum_{j\neq i}  (\gamma_i^2 - \gamma_j^2 ) c_{ij}^2 (t) v_i(t)  v_i(t)^\top \mathrm{d}t\right\|_F^2\nonumber\\
& = 3 I_2 + 3 I_2
\end{align}
where, for convenience, we have define
\begin{equation*}
    I_1:= \left\|\dfrac{1}{2}\int_0^T\sum_{i=1}^d \sum_{j \neq i} |\gamma_i^2 - \gamma_j^2 | |c_{ij} (t)|  \mathrm{d}\beta_{ji} (t)  ( v_i(t)  v_j(t)^\top + v_j(t)  v_i(t)^\top )\right\|_F^2
\end{equation*}
and
\begin{equation*}
    I_2:=  \left \|\int_0^T\sum_{i=1}^d \sum_{j\neq i}  (\gamma_i^2 - \gamma_j^2 ) c_{ij}^2 (t) v_i(t)  v_i(t)^\top \mathrm{d}t\right\|_F^2.
\end{equation*}

To evaluate the first integral $I_1$, define $X(t) = \int_0^T\sum_{i=1}^d \sum_{j \neq i} |\gamma_i^2 - \gamma_j^2 | |c_{ij} (t)|  \mathrm{d}\beta_{ji} (t) ( v_i(t)  v_j(t)^\top + v_j(t)  v_i(t)^\top ) $, we have that 
$$
\mathrm{d}X(t) = \sum_{i=1}^d \sum_{j \neq i} |\gamma_i^2 - \gamma_j^2 | |c_{ij} (t)|  \mathrm{d}\beta_{ji} (t)  ( v_i(t)  v_j(t)^\top + v_j(t)  v_i(t)^\top ) :=\sum_{i=1}^d \sum_{j \neq i}R_{ji}(t)\mathrm{d}\beta_{ji} (t)
$$
where $R_{ji}(t) := |\gamma_i^2 - \gamma_j^2 |\times |c_{ij} (t)|\times ( v_i(t)  v_j(t)^\top + v_j(t)  v_i(t)^\top )$,
so its $[l, r]$ component is
$$
\mathrm{d}X(t) [l, r] =\sum_{i=1}^d \sum_{j \neq i} R_{ji}(t)[l, r]\mathrm{d}\beta_{ji} (t).
$$

Defining the function $f(X) :=\|X\|_F^2:=\sum_{l = 1}^d\sum_{r = 1}^d X^2[l, r]$ and applying Ito's Lemma (Lemma \ref{Lemma_Ito}), we have
\begin{align}
    \mathrm{d} f(X) &=\sum_{l = 1}^d\sum_{r = 1}^d 2X(t)[l, r]\mathrm{d}X(t)[l, r] + \dfrac{1}{2}\sum_{l = 1}^d\sum_{r = 1}^d 2<\mathrm{d}X(t)[l, r], \mathrm{d}X(t)[l, r]> \nonumber \\
    &=\sum_{l = 1}^d\sum_{r = 1}^d 2X(t)[l, r]\sum_{i=1}^d \sum_{j \neq i} R_{ji}(t)[l, r]\mathrm{d}\beta_{ji} (t)  + \sum_{l = 1}^d\sum_{r = 1}^d\sum_{i=1}^d \sum_{j \neq i} R_{ji}^2(t)[l, r]\mathrm{d}t. \nonumber
\end{align}
Thus,
\begin{align}\label{eq:I1}
    \mathbb{E}(I_1 \times 1_E) &= \mathbb{E}\left[(f(X(T)) - f(X(0))\times 1_E\right] =  0 + \mathbb{E}[\int_0^T \sum_{l = 1}^d\sum_{r = 1}^d\sum_{i=1}^d \sum_{j \neq i} R_{ji}^2(t)[l, r]\mathrm{d}t \times 1_E] \nonumber \\
    &= \mathbb{E}[\int_0^T \sum_{l = 1}^d\sum_{r = 1}^d\sum_{i=1}^d \sum_{j \neq i} \left(  |\gamma_i^2 - \gamma_j^2 | |c_{ij} (t)|   ( v_i(t)  v_j(t)^\top + v_j(t)  v_i(t)^\top )[l, r]\right)^2\mathrm{d}t\times 1_E] \nonumber \\
    &= \mathbb{E}[\int_0^T \sum_{i=1}^d \sum_{j \neq i} \sum_{l = 1}^d\sum_{r = 1}^d\left(  |\gamma_i^2 - \gamma_j^2 | |c_{ij} (t)|  ( v_i(t)  v_j(t)^\top + v_j(t)  v_i(t)^\top )[l, r]\right)^2\mathrm{d}t\times 1_E] \nonumber \\
    &= \mathbb{E}[\int_0^T \sum_{i=1}^d \sum_{j \neq i}  \| |\gamma_i^2 - \gamma_j^2 | |c_{ij} (t)| ( v_i(t)  v_j(t)^\top + v_j(t)  v_i(t)^\top ))\|_F^2\mathrm{d}t\times 1_E] \nonumber \\
    &=2\int_0^T \mathbb{E}[ \sum_{i=1}^d \sum_{j \neq i}  (\gamma_i^2 - \gamma_j^2 )^2 c_{ij}^2 (t)  \|( v_i(t)  v_j(t)^\top + v_j(t)  v_i(t)^\top ))\|_F^2\mathrm{d}t\times 1_E] \nonumber \\
    &\le 2\int_0^T \mathbb{E}[ \sum_{i=1}^d \sum_{j \neq i}  \dfrac{16(\gamma_i^2 - \gamma_j^2 )^2 }{(\sigma_i - \sigma_j)^2}  2\mathrm{d}t] =64\int_0^T \mathbb{E}[ \sum_{i=1}^d \sum_{j \neq i}   \dfrac{(\gamma_i^2 - \gamma_j^2 )^2}{(\sigma_i - \sigma_j)^2}  \mathrm{d}t],
\end{align}
where the last inequality holds since, whenever the event $E$ occurs, we have $|\sigma_i (t) - \sigma_j (t)| \ge \dfrac{1}{2}(\sigma_i - \sigma_j)$ for any $i<j$ and any $t \in [0,T]$.

For the second integral $I_2$, we have
\begin{align}\label{eq:I2_b}
    I_2 &= \left \|\int_0^T\sum_{i=1}^d \sum_{j\neq i}  (\gamma_i^2 - \gamma_j^2 ) c_{ij}^2 (t) v_i(t)  v_i(t)^\top \mathrm{d}t\right\|_F^2\nonumber\\
    &= \left \|\int_0^T\sum_{i=1}^d \sum_{j\neq i}  (\gamma_i^2 - \gamma_j^2 ) c_{ij}^2 (t) v_i(t)  v_i(t)^\top \times 1  \, \, \mathrm{d}t\right\|_F^2\nonumber\\
& \leq  \int_0^T\left \|\sum_{i=1}^d \sum_{j\neq i}  (\gamma_i^2 - \gamma_j^2 ) c_{ij}^2 (t) v_i(t)  v_i(t)^\top \right\|_F^2\mathrm{d}t  \times \int_0^T 1^2 \mathrm{d}t\nonumber\\
& =  T\int_0^T\sum_{i=1}^d \left \|\sum_{j\neq i}  (\gamma_i^2 - \gamma_j^2 ) c_{ij}^2 (t) v_i(t)  v_i(t)^\top \right\|_F^2\mathrm{d}t\nonumber\\
& =  T\int_0^T\sum_{i=1}^d  \left(\sum_{j\neq i}  (\gamma_i^2 - \gamma_j^2 ) c_{ij}^2 (t)\right)^2 \|v_i(t)  v_i(t)^\top \|_F^2\mathrm{d}t\nonumber\\
& =  T\int_0^T\sum_{i=1}^d  \left(\sum_{j\neq i}  (\gamma_i^2 - \gamma_j^2 ) c_{ij}^2 (t)\right)^2\mathrm{d}t
\end{align}
where the first inequality holds by the Cauchy-Schwartz inequality, the third and fourth equalities hold since $v_i(t)v_i(t)^\top v_j(t)v_j(t)^\top = 0$ for all $i \neq j$.

Thus, since whenever the event $E$ occurs we have $|\sigma_i (t) - \sigma_j (t)| \ge \dfrac{1}{2}(\sigma_i - \sigma_j)$ for any $i<j$ and any $t \in [0,T]$, \eqref{eq:I2_b} implies that
\begin{equation}\label{eq:I2}
    I_2 \times \mathbbm{1}_E \leq 16 T \int_0^T \sum_{i=1}^d \left(\sum_{j\neq i} \dfrac{(\gamma_i^2 - \gamma_j^2 ) }{(\sigma_i - \sigma_j)^2}\right)^2\mathrm{d}t.
\end{equation}

We can express $\mathbb{E}[\|\Psi(T) - \Psi(0)\|_F^2]$ as the following sum,
\begin{equation}\label{eq_a1}
     \mathbb{E}[\|\Psi(T) - \Psi(0)\|_F^2] =  \mathbb{E}[\|\Psi(T) - \Psi(0)\|_F^2 \times 1_E] + \mathbb{E}[\|\Psi(T) - \Psi(0)\|_F^2 \times 1_{E^c}]  
     \end{equation}

     Combining \eqref{eq:I1} and \eqref{eq:I2}, it follows that
\begin{align} \label{eq_a2}    
     \mathbb{E}[\|\Psi(T) - \Psi(0)\|_F^2 \times 1_E] &\leq \mathbb{E}[\|\Psi(T) - \Psi(0)\|_F^2] \nonumber\\
     &\leq \mathbb{E}[\dfrac{1}{2}I_1 \times 1_E + I_2 \times 1_E] \nonumber\\
     &= \dfrac{1}{2}\mathbb{E}[I_1 \times 1_E] + \mathbb{E}[I_2 \times 1_E] \nonumber \\
     &\le 32\int_0^T \mathbb{E}[ \sum_{i=1}^d \sum_{j \neq i}   \dfrac{(\gamma_i^2 - \gamma_j^2 )^2}{(\sigma_i - \sigma_j)^2}  \mathrm{d}t] + 16T \int_0^T \mathbb{E}[ \sum_{i=1}^d \left(\sum_{j\neq i} \dfrac{(\gamma_i^2 - \gamma_j^2 ) }{(\sigma_i - \sigma_j)^2}\right)^2]\mathrm{d}t.  
\end{align}

Moreover, we have
\begin{align}\label{eq_a3}
\mathbb{E}[\|\Psi(T) - \Psi(0)\|_F^2 \times 1_{E^c}]&\leq \mathbb{P}(E^c) \nonumber\\
&\leq \mathbb{E}[4\|\Psi(T)\|_F^2 + 4\|\Psi(0)\|_F^2 \times 1_{E^c}]\nonumber\\
&\leq 8d \gamma_1^2 \mathbb{P}(E^c)\nonumber\\
&\leq 8d \gamma_1^2  \times \delta\nonumber\\
&\leq \frac{\gamma_1^2-\gamma_d^2 }{(\sigma_1-\sigma_d)^2} \nonumber\\
  &\le 32\int_0^T \mathbb{E}[ \sum_{i=1}^d \sum_{j \neq i}   \dfrac{(\gamma_i^2 - \gamma_j^2 )^2}{(\sigma_i - \sigma_j)^2}  \mathrm{d}t] + 16T \int_0^T \mathbb{E}[ \sum_{i=1}^d \left(\sum_{j\neq i} \dfrac{(\gamma_i^2 - \gamma_j^2 ) }{(\sigma_i - \sigma_j)^2}\right)^2]\mathrm{d}t, 
\end{align}
where the fifth inequality holds since $\delta \leq \frac{1}{8d \gamma_1^2} \times  \frac{\gamma_1^2-\gamma_d^2 }{(\sigma_1-\sigma_d)^2}$.

Therefore, plugging \eqref{eq_a2} and \eqref{eq_a3} into \eqref{eq_a1}, we have
\begin{align}
     \mathbb{E}[\|\Psi(T) - \Psi(0)\|_F^2] \le 64\int_0^T \mathbb{E}[ \sum_{i=1}^d \sum_{j \neq i}   \dfrac{(\gamma_i^2 - \gamma_j^2 )^2}{(\sigma_i - \sigma_j)^2}  \mathrm{d}t] + 32T \int_0^T \mathbb{E}[ \sum_{i=1}^d \left(\sum_{j\neq i} \dfrac{(\gamma_i^2 - \gamma_j^2 ) }{(\sigma_i - \sigma_j)^2}\right)^2]\mathrm{d}t.  \nonumber
\end{align}

\end{proof}

\subsection{Completing the proof of Theorem \ref{ThCovariance}}\label{main-proof}

We now complete the proof of the main result.

\begin{proof}[Proof of Theorem \ref{ThCovariance}]

    From Lemma \ref{boundofz}, we have 
\begin{align}\label{eq_z1}
    &\mathbb{E}\left[\| \hat{V}\Gamma^\top\Gamma\hat{V}^\top - V\Gamma^\top\Gamma V^\top \|_F^2\right] = \mathbb{E}\left[\| \Psi (T) - \Psi (0) \|_F^2\right]\nonumber\\
    & \le 32\int_0^T \mathbb{E}[ \sum_{i=1}^d \sum_{j \neq i}   \dfrac{(\gamma_i^2 - \gamma_j^2 )^2}{(\sigma_i - \sigma_j)^2}  \mathrm{d}t] + 16T \int_0^T \mathbb{E}[ \sum_{i=1}^d \left(\sum_{j\neq i} \dfrac{(\gamma_i^2 - \gamma_j^2 ) }{(\sigma_i - \sigma_j)^2}\right)^2]\mathrm{d}t  \nonumber \\
    & \le 64\int_0^T \mathbb{E}[ \sum_{i=1}^d \sum_{j = i+1}^d   \dfrac{(\gamma_i^2 - \gamma_j^2 )^2}{(\sigma_i - \sigma_j)^2}  \mathrm{d}t] + 32T \int_0^T \mathbb{E}[ \sum_{i=1}^d \left(\sum_{j = i+1}^d \dfrac{(\gamma_i^2 - \gamma_j^2 ) }{(\sigma_i - \sigma_j)^2}\right)^2]\mathrm{d}t  \nonumber \\
    & = 64\int_0^T \mathbb{E}[ \sum_{i=1}^k \sum_{j = i+1}^d   \dfrac{(\gamma_i^2 - \gamma_j^2 )^2}{(\sigma_i - \sigma_j)^2}  \mathrm{d}t] + 32T \int_0^T \mathbb{E}[ \sum_{i=1}^k \left(\sum_{j = i+1}^d \dfrac{(\gamma_i^2 - \gamma_j^2 ) }{(\sigma_i - \sigma_j)^2}\right)^2]\mathrm{d}t  \nonumber \\
    & = 64T   \sum_{i=1}^k \sum_{j = i+1}^d   \dfrac{(\gamma_i^2 - \gamma_j^2 )^2}{(\sigma_i - \sigma_j)^2}   + 32T^2   \sum_{i=1}^k \left(\sum_{j = i+1}^d \dfrac{(\gamma_i^2 - \gamma_j^2 ) }{(\sigma_i - \sigma_j)^2}\right)^2 \nonumber \\
    & = O\left(\sum_{i=1}^k \sum_{j = i+1}^d   \dfrac{(\gamma_i^2 - \gamma_j^2 )^2}{(\sigma_i - \sigma_j)^2}   + T \sum_{i=1}^k \left(\sum_{j = i+1}^d \dfrac{(\gamma_i^2 - \gamma_j^2 ) }{(\sigma_i - \sigma_j)^2}\right)^2 \right)T. 
\end{align}
By the Cauchy-Schwarz inequality, we have that
\begin{align}\label{eq_z2}
    \left(\sum_{j = i+1}^d \dfrac{(\gamma_i^2 - \gamma_j^2 ) }{(\sigma_i - \sigma_j)^2}\right)^2 &= \left(\sum_{j = i+1}^d \dfrac{1 }{|\sigma_i - \sigma_j|} \times\dfrac{|\gamma_i^2 - \gamma_j^2|}{|\sigma_i - \sigma_j|}\right)^2 \nonumber \\
    &\le \left(\sum_{j = i+1}^d \dfrac{1 }{(\sigma_i - \sigma_j)^2}\right) \times \left(\sum_{j = i+1}^d \dfrac{(\gamma_i^2 - \gamma_j^2 )^2 }{(\sigma_i - \sigma_j)^2}\right) \nonumber \\
    &\le \left(\sum_{j = i+1}^d \dfrac{1 }{(\sqrt{d})^2}\right) \times \left(\sum_{j = i+1}^d \dfrac{(\gamma_i^2 - \gamma_j^2 )^2 }{(\sigma_i - \sigma_j)^2}\right) \nonumber \\
    &\le \sum_{j = i+1}^d \dfrac{(\gamma_i^2 - \gamma_j^2 )^2 }{(\sigma_i - \sigma_j)^2}.
\end{align}

Plugging \eqref{eq_z2}  into \eqref{eq_z1}, we have
\begin{equation*}
    \mathbb{E}\left[\| \hat{V}\Gamma^\top\Gamma\hat{V}^\top - V\Gamma^\top\Gamma V^\top \|_F^2\right] \le O\left(\sum_{i=1}^k \sum_{j = i+1}^d   \dfrac{(\gamma_i^2 - \gamma_j^2 )^2}{(\sigma_i - \sigma_j)^2}    \right)T .
\end{equation*}
\end{proof}

\section{Proof of Corollary \ref{Subspace}}\label{sec_proof_subspace}

\begin{proof}[Proof of Corollary \ref{Subspace}]
    To prove Corollary \ref{Subspace}, we plug in $\gamma_1 = \cdots = \gamma_k = 1$ and $\gamma_{k+1} = \cdots = \gamma_d = 0$ to Theorem \ref{ThCovariance}. There are two cases. 

In the first case, where $A$ may be any $m\times d$ matrix which satisfies Assumption \ref{assumption}, plugging in $\gamma_1 = \cdots = \gamma_k = 1$ and $\gamma_{k+1} = \cdots = \gamma_d = 0$ to Theorem \ref{ThCovariance} we get
\begin{align}
    \mathbb{E}\left[\| \hat{V}_k\hat{V}_k^\top - V_k V_k^\top \|_F^2\right] &= \mathbb{E} \left[\| \hat{V} \Gamma^\top\Gamma\hat{V}^\top - V\Gamma^\top\Gamma V^\top \|_F^2\right] \nonumber \\ 
    & \le O\left( \sum_{i=1}^k \sum_{j=i+1}^d \dfrac{(\gamma_i^2-\gamma_j^2)^2}{(\sigma_i - \sigma_j)^2}\right) T \nonumber \\ 
    & = O\left( \sum_{i=1}^k \sum_{j=k+1}^d \dfrac{1}{(\sigma_i - \sigma_j)^2}\right) T \nonumber \\ 
    & \le O\left( \sum_{i=1}^k \sum_{j=k+1}^d \dfrac{1}{(\sigma_k - \sigma_{k+1})^2}\right) T \nonumber \\ 
    & \le O\left( \dfrac{kd}{(\sigma_k - \sigma_{k+1})^2}T\right)
\end{align}  
where the first inequality holds by Theorem \ref{ThCovariance} and the second equality holds in that $\gamma_1 = \cdots = \gamma_k = 1$ and $\gamma_{k+1} = \cdots = \gamma_d = 0$.

By Jensen's Inequality, we have that
\begin{equation*} 
   \mathbb{E}\left[\| \hat{V}_k\hat{V}_k^\top - V_k V_k^\top \|_F\right]\le \sqrt{ \mathbb{E}\left[\| \hat{V}_k\hat{V}_k^\top - V_k V_k^\top \|_F^2\right]} \le O(\dfrac{\sqrt{kd}}{(\sigma_k - \sigma_{k+1})})\sqrt{T}.
\end{equation*}

In the second case, where the singular values of $A$ also satisfy $\sigma_i - \sigma_{i+1} \ge \Omega(\sigma_k - \sigma_{k+1})$ for all $i \le k$, we have 
\begin{align}
    \mathbb{E}\left[\| \hat{V}_k\hat{V}_k^\top - V_k V_k^\top \|_F^2\right] &= \mathbb{E} \left[\| \hat{V} \Gamma^\top\Gamma\hat{V}^\top - V\Gamma^\top\Gamma V^\top \|_F^2\right] \nonumber \\ 
    & \le O\left( \sum_{i=1}^k \sum_{j=i+1}^d \dfrac{(\gamma_i^2-\gamma_j^2)^2}{(\sigma_i - \sigma_j)^2}\right) T \nonumber \\ 
    & = O\left( \sum_{i=1}^k \sum_{j=k+1}^d \dfrac{1}{(\sigma_i - \sigma_j)^2}\right) T \nonumber \\ 
    & \le O\left( \sum_{i=1}^k \sum_{j=k+1}^d \dfrac{1}{(i-k-1)^2(\sigma_k - \sigma_{k+1})^2}\right) T \nonumber \\ 
    & \le O\left( \sum_{i=1}^k  \dfrac{d}{(i-k-1)^2(\sigma_k - \sigma_{k+1})^2}\right) T \nonumber \\ 
    & \le O\left(  \dfrac{d}{(\sigma_k - \sigma_{k+1})^2} \sum_{i=1}^k \dfrac{1}{i^2}  \right) T \nonumber \\ 
    & \le O\left( \dfrac{d}{(\sigma_k - \sigma_{k+1})^2}\right)T
\end{align}  
where the first inequality holds by Theorem \ref{ThCovariance} and the second equality holds since $\gamma_1 = \cdots = \gamma_k = 1$ and $\gamma_{k+1} = \cdots = \gamma_d = 0$, the second inequality holds since $\sigma_i - \sigma_{i+1} \ge \Omega(\sigma_k - \sigma_{k+1})$ for all $i \le k$, and the last inequality holds since $\sum_{i=1}^k \frac{1}{i^2} \le \sum_{i=1}^{\infty} \frac{1}{i^2} =O(1) $.

Thanks to Jensen's Inequality, we have that
\begin{equation*} 
   \mathbb{E}\left[\| \hat{V}_k\hat{V}_k^\top - V_k V_k^\top \|_F\right]\le \sqrt{ \mathbb{E}\left[\| \hat{V}_k\hat{V}_k^\top - V_k V_k^\top \|_F^2\right]} \le O(\dfrac{\sqrt{d}}{(\sigma_k - \sigma_{k+1})})\sqrt{T}.
\end{equation*}
\end{proof}

\section{Proof of Corollary \ref{rankkapprox}} \label{sec_proof_covariance}

\begin{proof}[Proof of Corollary \ref{rankkapprox}]
We first bound the quantity $\mathbb{E}\left[\| \hat{V}{\Sigma}_k^\top\Sigma_k\hat{V}^\top - V\Sigma_k^\top\Sigma_k V^\top \|_F\right]$.

Set $\gamma_i = \sigma_i$ for $i \leq k$ and $\gamma_i = 0$ for $i >k$.
Then by Theorem \ref{ThCovariance} we have
\begin{align}
\label{eq:27}
   &\mathbb{E}\left[\| \hat{V}\Sigma_k^\top\Sigma_k\hat{V}^\top - V\Sigma_k^\top\Sigma_k V^\top \|_F^2\right] \le O\left(\sum_{i=1}^k \sum_{j=i+1}^d \dfrac{(\gamma_i^2 - \gamma_j^2)^2}{(\sigma_i - \sigma_j)^2}\right)T \nonumber\\
   &= O\left(\sum_{i=1}^{k-1} \sum_{j=i+1}^k \dfrac{(\sigma_i^2 - \sigma_j^2)^2}{(\sigma_i - \sigma_j)^2} + \sum_{i=1}^{k} \sum_{j=k+1}^d \left(\dfrac{\sigma_i^2 - 0^2}{\sigma_i - \sigma_j}\right)^2\right)T \nonumber\\
    &= O\left(\sum_{i=1}^{k-1} \sum_{j=i+1}^k (\sigma_i + \sigma_j)^2 + \sum_{i=1}^{k} \sum_{j=k+1}^d  \left(\dfrac{\sigma_i^2 - \sigma_k^2}{\sigma_i - \sigma_j} + \dfrac{\sigma_k^2}{\sigma_i - \sigma_j}\right)^2\right)T \nonumber\\
    &\leq O\left(\sum_{i=1}^{k-1} \sum_{j=i+1}^k (\sigma_i + \sigma_j)^2 + \sum_{i=1}^{k} \sum_{j=k+1}^d  \left(\sigma_i + \dfrac{\sigma_k^2}{\sigma_i - \sigma_j}\right)^2\right)T \nonumber\\
    &= O\left(\sum_{i=1}^{k-1} \sum_{j=i+1}^k (\sigma_i + \sigma_j)^2 + \sum_{i=1}^{k} \sum_{j=k+1}^d  \left(\sigma_i + \sigma_k \dfrac{\sigma_k}{\sigma_i - \sigma_j}\right)^2\right)T \nonumber\\
     &\leq O\left(\sum_{i=1}^{k-1} \sum_{j=i+1}^k (\sigma_i + \sigma_j)^2 + \sum_{i=1}^{k} \sum_{j=k+1}^d  \left(\sigma_i + \sigma_k \dfrac{\sigma_k}{\sigma_k - \sigma_{j}}\right)^2\right)T \nonumber\\
    &\leq O\left(\sum_{i=1}^{k-1} \sum_{j=i+1}^k (\sigma_i + \sigma_j)^2 + \sum_{i=1}^{k} \sum_{j=k+1}^d  \sigma_i^2 +  \sum_{i=1}^{k} \sum_{j=k+1}^d \left(\sigma_k \dfrac{\sigma_k}{\sigma_k - \sigma_{j}}\right)^2\right)T \nonumber\\
    &\leq O\left( d\|\Sigma_k\|_F^2 +  \sum_{i=1}^{k} \sum_{j=k+1}^d \left(\sigma_k \dfrac{\sigma_k}{\sigma_k - \sigma_{j}}\right)^2\right)T \nonumber\\
    &\leq O\left( d\|\Sigma_k\|_F^2 +  k(d-k) \left(\sigma_k \dfrac{\sigma_k}{\sigma_k - \sigma_{k+1}}\right)^2\right)T.
\end{align}

We next bound the quantity $\mathbb{E}\left[\| \hat{V}\hat{\Sigma}_k^\top\hat{\Sigma}_k\hat{V}^\top - \hat{V}\Sigma_k^\top\Sigma_k \hat{V}^\top \|_F\right] $.

Let $E_1$ be the event when $\|G\| > \sqrt{\max(m,d)} \log(1/\delta)$.
By Lemma \ref{lemma_concentration}, we have $\mathbb{P}(E_1) \geq 1- \delta$.

Since  $\|\Sigma_k\|_F \leq \sqrt{k} \sigma_1 $ and $\|\hat{\Sigma}_k\|_F < \sqrt{k}\sigma_1(t)$, we can use the bound $$\|\Sigma_k^\top\Sigma_k - \hat{\Sigma}_k^\top\hat{\Sigma}_k\|_F < \|\Sigma_k^\top\Sigma_k\|_F + \|\hat{\Sigma}_k^\top\hat{\Sigma}_k\|_F < k \sigma_1 + k\sigma_1^2(t) < 4k\sigma_1^2$$ and hence
$$\mathbb{E}[\|\Sigma_k^\top\Sigma_k - \hat{\Sigma}_k^\top\hat{\Sigma}_k\|_F * 1_{E_1}] < 2\sqrt{k}\sigma_1 * P(E_1)<4k\sigma_1^2 *\delta.$$

Recall that (from Assumption \ref{assumption}) $\delta < \frac{1}{k\sigma_1^2}$.  Hence,
$$\mathbb{E}[\|\Sigma_k^\top \Sigma_k  - \hat{\Sigma}_k^\top\hat{\Sigma}_k\|_F * 1_{E_1}] <4 $$
Now consider the event $E_1^c$, where $\|G\| < \sqrt{\max(m,d)} \log(1/\delta)$.
From above, we have $\mathbb{P}(E_1^c) = 1- \mathbb{P}(E_1) \leq \delta$.
For $E_1^c$ we get,
\begin{align*}
\mathbb{E}[\|\Sigma_k^\top\Sigma_k  - \hat{\Sigma}_k^\top\hat{\Sigma}_k\|_F * 1_{E_1^c}] &< \mathbb{E}[\|(\Sigma_k  - \hat{\Sigma}_k)(\Sigma_k  + \hat{\Sigma}_k)\|_F * 1_{E_1^c}]\\
&< \mathbb{E}[\sqrt{T}\|G_k\| *(\|\Sigma_k\|_F + \|\hat{\Sigma}_k\|_F )* 1_{E_1^c}]\\
&< \mathbb{E}[2\sqrt{kT}\sigma_1\|G_k\|  * 1_{E_1^c}]\\
&< 2\sqrt{kd}\sigma_1\log(1/\delta)\sqrt{T}.
\end{align*}

Finally, put the two cases together:
\begin{align}
\label{eq:28}
     \mathbb{E}\left[\| \hat{V}\hat{\Sigma}_k^\top\hat{\Sigma}_k\hat{V}^\top - \hat{V}\Sigma_k^\top\Sigma_k \hat{V}^\top \|_F\right] 
 &= \mathbb{E}[\|\Sigma_k  - \hat{\Sigma}_k\|_F ]  \nonumber\\
 &=   \mathbb{E}[\|\Sigma_k  - \hat{\Sigma}_k\|_F * 1_{E_1}]  + \mathbb{E}[\|\Sigma_k  - \hat{\Sigma}_k\|_F * 1_{E_1^c}] \nonumber\\
 & < 4 + 2\sqrt{kd}\sigma_1 \log(1/\delta) \sqrt{T} \nonumber\\
  & = O(\sqrt{kd}\sigma_1 \log(1/\delta))\sqrt{T}.
\end{align}

Combining \eqref{eq:27} and \eqref{eq:28}, we have
\begin{align}
    \mathbb{E} & \left[\| \hat{V}\hat{\Sigma}_k^\top\hat{\Sigma}_k\hat{V}^\top - V\Sigma_k^\top\Sigma_k V^\top \|_F\right] \nonumber\\
    & \le 
   \mathbb{E}\left[\| \hat{V}\hat{\Sigma}_k^\top\hat{\Sigma}_k\hat{V}^\top - \hat{V}\Sigma_k^\top\Sigma_k \hat{V}^\top \|_F\right] + \mathbb{E}\left[\| \hat{V}\Sigma_k^\top\Sigma_k\hat{V}^\top - V\Sigma_k^\top\Sigma_k V^\top \|_F\right] \nonumber \\ 
   & \leq O\left( \sqrt{d}\|\Sigma_k\|_F +  \sqrt{k(d-k)} \left(\sigma_k \dfrac{\sigma_k}{\sigma_k - \sigma_{k+1}}\right)\right)\sqrt{T} + O(\sqrt{kd}\sigma_1 \log(1/\delta)) \sqrt{T}\nonumber \\
   &\leq O\left( \sqrt{d}\|\Sigma_k\|_F +  \sqrt{k(d-k)} \left(\sigma_k \dfrac{\sigma_k}{\sigma_k - \sigma_{k+1}}\right)\right)\sqrt{T}.
\end{align} 
\end{proof}

\section{Additional comparisons for low-rank covariance approximation}\label{sec_covariance_baselines}

In this section, we present how one can derive high-probability bounds on the quantity $\|\hat{V} \hat{\Sigma}_k^2 \hat{V}^\top - V \Sigma_k^2 V^\top\|_F$ from the subspace perturbation bounds of \cite{davis1970rotation, wedin1972perturbation} or \cite{o2018random}.  

Towards this end, we note that
$$
 \|\hat{V} \hat{\Sigma}_k^2 \hat{V}^\top - V \Sigma_k^2 V^\top\|_F \leq  \|\hat{V} \hat{\Sigma}_k^2 \hat{V}^\top - \hat{V} \Sigma_k^2 \hat{V}^\top\|_F +   \|\hat{V} \Sigma_k^2 \hat{V}^\top  - V \Sigma_k^2 V^\top\|_F.
$$

The first term can be bounded as 
$$\|\hat{V} \hat{\Sigma}_k^2 \hat{V}^\top - \hat{V} \Sigma_k^2 \hat{V}^\top\|_F =\| \hat{\Sigma}_k^2 - \Sigma_k^2\|_F = \sum_{i=1}^k \hat{\sigma}_i^2 - \sigma_i^2,$$
which can be bounded using Weyl's inequality (Lemma \ref{lemma_weyl}) together with the Gaussian concentration inequality in Lemma \ref{lemma_concentration}.

For the second term, we have
\begin{align}\label{eq_d1}
    \|\hat{V} \Sigma_k^2 \hat{V}^\top - V \Sigma_k^2 V^\top\|_F &=  \|\hat{V} \Sigma_k^2 \hat{V}^\top - V \Sigma_k^2 V^\top\|_F \nonumber\\
&=    \|\sum_{i=1}^{k-1} (\sigma_i^2 - \sigma_{i+1}^2) (\hat{V}_i \hat{V}_i^\top - V_i V_i^\top) + \sigma_k^2 (\hat{V}_k \hat{V}_k^\top - V_k V_k^\top)\|_F \nonumber\\
    &\leq \sum_{i=1}^{k-1} (\sigma_i^2 - \sigma_{i+1}^2) \|\hat{V}_i \hat{V}_i^\top - V_i V_i^\top\|_F + \sigma_k^2 \|\hat{V}_k \hat{V}_k^\top - V_k V_k^\top\|_F \nonumber\\
        &\leq \sum_{i=1}^{k-1} (\sigma_i + \sigma_{i+1})(\sigma_i - \sigma_{i+1}) \|\hat{V}_i \hat{V}_i^\top - V_i V_i^\top\|_F + \sigma_k^2 \|\hat{V}_k \hat{V}_k^\top - V_k V_k^\top\|_F \nonumber\\
    & \leq \sum_{i=1}^{k-1} (\sigma_i - \sigma_{i+1}) \frac{\sqrt{i} \sqrt{d}}{\sigma_i - \sigma_{i+1}} + \sigma_k   \frac{\sqrt{k} \sqrt{d}}{\sigma_k - \sigma_{k+1}} \nonumber\\
    & = O\left(k^{1.5} \sqrt{d} + \frac{\sigma_k}{\sigma_k - \sigma_{k+1}} \sqrt{k} \sqrt{d}\right).  
\end{align}

Plugging into \eqref{eq_d1} the bound of $\| V_i V_i^\top - \hat{V}_i \hat{V}_i^\top \|_F \leq  \frac{\sqrt{i} \sqrt{m}}{\sigma_i-\sigma_{i+1}} \sqrt{T}$ w.h.p. implied by \cite{davis1970rotation, wedin1972perturbation}, one has
\begin{align}
    \|\hat{V} \Sigma_k^2 \hat{V}^\top - V \Sigma_k^2 V^\top\|_F 
    &\leq \sum_{i=1}^{k-1} (\sigma_i + \sigma_{i+1})(\sigma_i - \sigma_{i+1}) \frac{\sqrt{i} \sqrt{m}}{\sigma_i-\sigma_{i+1}} \sqrt{T} + \sigma_k^2 \frac{\sqrt{k} \sqrt{m}}{\sigma_k-\sigma_{k+1}} \sqrt{T} \nonumber\\
        &\leq  \sqrt{m} \sqrt{T} \sum_{i=1}^{k-1} (\sigma_i + \sigma_{i+1}) \sqrt{i} + \sigma_k^2 \frac{\sqrt{k} \sqrt{m}}{\sigma_k-\sigma_{k+1}} \sqrt{T} \nonumber\\
                &\leq  2k^{1.5}\sqrt{m} \sqrt{T}  \sigma_1 + \sigma_k^2 \frac{\sqrt{k} \sqrt{m}}{\sigma_k-\sigma_{k+1}} \sqrt{T}. \nonumber
\end{align}

One can also instead plug in the bound from Theorem 18 of \cite{o2018random} (restated here as \eqref{Vanvu_result} in Section \ref{sec_related_work}) into \eqref{eq_d1}.
 When, e.g., $\sigma_k- \sigma_{k+1} \geq \Omega(\max(\sigma_k, \sqrt{m}))$, \eqref{Vanvu_result} reduces to  $\|\hat{V} \Sigma_i^2 \hat{V}^\top - V \Sigma_i^2 V^\top\|_F \leq  O\left(i\frac{\sqrt{m}}{\sigma_i} \sqrt{T}\right)$  for $i \leq k$ into \eqref{eq_d1}.
 Thus, plugging in this bound into \eqref{eq_d1}, one has
 
 \begin{align*}
    \|\hat{V} \Sigma_k^2 \hat{V}^\top - V \Sigma_k^2 V^\top\|_F 
        &\leq \sum_{i=1}^{k-1} (\sigma_i + \sigma_{i+1})(\sigma_i - \sigma_{i+1}) i\frac{\sqrt{m}}{\sigma_i}\sqrt{T} + \sigma_k^2 k\frac{\sqrt{m}}{\sigma_k}\sqrt{T}\nonumber\\
                &\leq O\left(\sum_{i=1}^{k-1} (\sigma_i - \sigma_{i+1}) i\sqrt{m}\sqrt{T} + \sigma_k k\sqrt{m}\sqrt{T}\right)\nonumber\\
                                &\leq O\left((\sigma_1 - \sigma_{k}) k\sqrt{m}\sqrt{T} + \sigma_k k\sqrt{m}\sqrt{T}\right)\nonumber\\
                                &\leq O\left(\sigma_1 k\sqrt{m}\sqrt{T}\right).
\end{align*}

\newpage

\bibliography{DP}
\bibliographystyle{plain}

\end{document}